\documentclass[12pt,reqno]{amsart} 
\usepackage{amsaddr}
\usepackage{amsmath,amsthm,amssymb,amscd,faktor,float,mathdots,mathtools,stmaryrd,enumerate,shuffle,hyperref,xcolor,ytableau,tabularx,tikz}

\RequirePackage{babel} 

\RequirePackage{scrbase} 

\RequirePackage{scrhack} 

\RequirePackage{setspace} 

\RequirePackage{longtable} 

\RequirePackage{siunitx} 

\RequirePackage{graphicx} 
\graphicspath{{Figures/}{./}} 
\usepackage{float}

\usepackage{fancyvrb}

\usepackage{listings}

\usepackage{pdfpages}

\usepackage{ragged2e}

\usepackage{enumitem}

\definecolor{codegreen}{rgb}{0,0.6,0}
\definecolor{codegray}{rgb}{0.5,0.5,0.5}
\definecolor{codepurple}{rgb}{0.58,0,0.82}
\definecolor{backcolour}{rgb}{0.95,0.95,0.92}
\definecolor{lightblue}{rgb}{0,0,0.65}

\lstdefinestyle{mystyle}{
    backgroundcolor=\color{backcolour},   
    commentstyle=\color{lightblue},
    keywordstyle=\color{codegreen},
    numberstyle=\tiny\color{codegray},
    stringstyle=\color{codepurple},
    basicstyle=\ttfamily\footnotesize,
    breakatwhitespace=false,         
    breaklines=true,                 
    captionpos=b,                    
    keepspaces=true,                 
    numbers=left,                    
    numbersep=5pt,                  
    showspaces=false,                
    showstringspaces=false,
    showtabs=false,                  
    tabsize=2
}

\RequirePackage{booktabs} 

\RequirePackage{caption} 
\mathtoolsset{showonlyrefs}

\hypersetup{pdfborder=0 0 0}
\hypersetup{colorlinks=true}
\hypersetup{citecolor=magenta}
\hypersetup{linkcolor=blue}
\hypersetup{urlcolor=cyan}
\usetikzlibrary{arrows,positioning}
\usetikzlibrary{calc}
\usetikzlibrary{decorations}
\usetikzlibrary{decorations.pathreplacing}
\tikzset{
    >=stealth',
    punkt/.style={
           rectangle,
           rounded corners,
           draw=black, very thick,
           text width=6.5em,
           minimum height=2em,
           text centered},
    pil/.style={
           ->,
           thick,
           shorten <=2pt,
           shorten >=2pt,}
}
\usepackage{wrapfig}
\usepackage{pgfplots}
\allowdisplaybreaks

\title{Combinatorial interpretation of the Schlesinger--Zudilin stuffle product}

\date{August 05, 2025.}
\author{Benjamin Brindle$^1$}
\address{$^1$Department of Mathematics, University of Hamburg,\\ Bundesstrasse 55, 20146 Hamburg, Germany.\\ $^1$benjamin.brindle@uni-hamburg.de}


\newtheorem{theorem}{Theorem}[section]
\newtheorem{proposition}[theorem]{Proposition}

\newtheorem{lemma}[theorem]{Lemma}

\newtheorem{conjecture}[theorem]{Conjecture}
\numberwithin{equation}{section}
\theoremstyle{definition}
\newtheorem{definition}[theorem]{Definition}
\newtheorem{example}[theorem]{Example}
\newtheorem{remark}[theorem]{Remark}

\newcommand{\Q}{\mathbb{Q}}
\newcommand{\Z}{\mathbb{Z}}

\newcommand{\dep}{\operatorname{depth}}

\newcommand{\len}{\operatorname{len}}

\newcommand{\word}{\mathtt{W}}

\newcommand{\smpart}[1]{\operatorname{sm}\left(#1\right)}
\newcommand{\cmark}[1]{\operatorname{C}\left(#1\right)}

\newcommand{\bone}{{\mathbf{1}}}

\newcommand{\sz}{\zeta_q^\mathrm{SZ}} 

\newcommand{\qmzv}{$q$MZV}

\newcommand{\dd}{d}
\newcommand{\dr}{r}

\newcommand{\one}{\boldsymbol{1}}

\newcommand{\df}{:=}

\newcommand{\QB}{\Q\langle \mathcal{U} \rangle}

\definecolor{dutchorange}{RGB}{220,50,0}
\definecolor{darkblue}{RGB}{0,71,171}

\definecolor{mygray}{RGB}{160,160,160}
\definecolor{mycyan}{RGB}{150,255,255}
\definecolor{darkred}{RGB}{130,0,0}

\tikzset{
    >=stealth',
    punkt/.style={
           rectangle,
           rounded corners,
           draw=black, very thick,
           text width=6.5em,
           minimum height=2em,
           text centered},
    pil/.style={
           ->,
           thick,
           shorten <=2pt,
           shorten >=2pt,}
}

\newcommand{\crpart}[3]{\ensuremath{
		\ytableausetup{smalltableaux}
		\begin{tikzpicture}[inner sep=0in,outer sep=0in,scale=0.34]
			\node (n) 
			{\ydiagram{#1}};
			\foreach \x in {#3}
			\filldraw[fill=red] (n.north west)+(0.5+\x-1,0.3) circle [radius=0.13cm];
			\foreach \y in {#2}	
			\filldraw[fill=green] (n.north west)+(-0.3,-0.5-\y+1) circle [radius=0.13cm];
\end{tikzpicture}}}

\makeatletter
\newcommand{\gettikzxy}[3]{%
  \tikz@scan@one@point\pgfutil@firstofone#1\relax
  \edef#2{\the\pgf@x}%
  \edef#3{\the\pgf@y}%
}
\makeatother

\begin{document}
\maketitle

\begin{abstract}
    We derive an explicit formula for the quasi--shuffle product satisfied by Schlesinger--Zudilin Multiple~$q$-Zeta Values, expressed in terms of partition data. To achieve this, we interpret Schlesinger--Zudilin Multiple~$q$-Zeta Values as generating series of distinguished marked partitions, which are partitions whose Young diagrams have certain rows and columns marked. Together with the description of duality using marked partitions in~\cite{Br2}, and Bachmann’s conjecture~(\cite{BaTalk}) that all linear relations among Multiple~$q$-Zeta Values are implied by duality and the stuffle product, this paper completes the description of the conjectural structure of Multiple~$q$-Zeta Values using marked partitions.
\end{abstract}

\section{Introduction}
Multiple~$q$-Zeta Values, or~$q$MZVs for short, can be regarded as generalizations of Multiple Zeta Values (MZVs) as well as (quasi-)modular forms or as generating functions of specific types of partitions. They are~$q$-series that reduce to a Multiple Zeta Value (or a~$\Q$-linear combination of them) in the limit~$q \to 1$, often after modifying the series by multiplying it with a suitable power of~$1-q$. In this paper, we focus on the~$q$MZVs introduced by Schlesinger \cite{Sch} and Zudilin \cite{Zu2}. For an overview of~$q$MZVs, see, for instance,~\cite{Br2}.

In the following, let us consider the set~$\mathcal{U} := \{u_j\mid j\in\Z_{\geq 0}\}$. We refer to~$\mathcal{U}$ as an \emph{alphabet}, and its elements are called \emph{letters}. Furthermore, monomials composed of elements of~$\mathcal{U}$ (with respect to concatenation) are called \emph{words}. The neutral element under concatenation is denoted by~$\bone$, which we call the \emph{empty word}. Let~$\mathcal{U}^\ast$ denote the set of words formed from letters in~$\mathcal{U}$. We define~$\Q\langle\mathcal{U}\rangle$ as the~$\Q$-vector space~$\operatorname{span}_\Q\mathcal{U}^\ast$, which is provided with the non-commutative, but associative multiplication given by concatenation. The \emph{Schlesinger--Zudilin stuffle product}~$\ast\colon\QB\times \QB\to \QB$ is defined via~$\Q$-bilinearity and the recursion
\begin{align*}
    u_{j_1} \word_1 \ast u_{j_2} \word_2 :=  u_{j_1} \left( \word_1 \ast u_{j_2} \word_2\right)+ u_{j_2} \left(u_{j_1} \word_1 \ast \word_2\right) + u_{j_1+j_2}\left(\word_1 \ast \word_2\right)
\end{align*}
for all~$j_1,j_2\in\Z_{\geq 0}$ and~$\word_1,\word_2 \in \mathcal{U}^\ast$, with initial condition~$\one \ast \word = \word \ast \one = \word$ for any~$\word \in \mathcal{U}^\ast$. Note that the usual stuffle product is defined on~$\Q\langle\mathcal{U} \backslash \{u_0\}\rangle \times \Q\langle\mathcal{U} \backslash \{u_0\}\rangle$ only, using the same recursion. Nevertheless, throughout this paper, we refer to the \emph{Schlesinger--Zudilin stuffle product} whenever we use the term \lq\lq stuffle product\rq\rq\, for consistency and better readability. By a result of Hoffman~(\cite{Ho}), it is known that~$(\QB, \ast)$ is an associative and commutative~$\Q$-algebra. Given a word~$\word = u_{k_1}\cdots u_{k_\dr}\in\mathcal{U}^{\ast}$, we call~$\len(\word) := \dr$  its \emph{length} and~$\dep(\word) := \#\{k_j\neq 0\mid 1\leq j\leq \dr\}$ its \emph{depth}. Let~$\mathcal{U}^{\ast,\circ} := \mathcal{U}^\ast \backslash u_0 \mathcal{U}^\ast$ be the set of words in~$\mathcal{U}^\ast$ that do not start with~$u_0$.  Furthermore, denote by~$\QB^\circ$ the corresponding subspace of~$\QB$, i.e., the~$\Q$-vector space generated by words not starting with~$u_0$. Note that~$\ast$ restricts to a map~$\QB^\circ\times \QB^\circ\to \QB^\circ$ which causes a commutative~$\Q$-algebra~$(\QB^\circ,\ast)$, as shown in~\cite{Ho}. Furthermore, the map~$\sz\colon (\QB^\circ,\ast)\to (\Q\llbracket q\rrbracket,\cdot)$ is the~$\Q$-algebra homomorphism (see~\cite{HI}) defined by~$\sz(\bone) = 1$,~$\Q$-linearity, and, with~$m_{\dd+1} := 0$,
\begin{align}
\label{eq:SZ-defAppmp}
    \sz\left(u_{k_1}u_0^{z_1}\cdots u_{k_\dd}u_0^{z_\dd}\right) := \sum\limits_{m_1>\cdots>m_\dd>0} \prod\limits_{j=1}^\dd \binom{m_j-m_{j+1}-1}{z_j} \frac{q^{m_j k_j}}{(1-q^{m_j})^{k_j}},
\end{align}
for any~$k_1,\dots,k_\dd\in\Z_{>0}$ and~$z_1,\dots,z_\dd\in\Z_{\geq 0}$ where~$\dd\in\Z_{>0}$ (note that this definition is not the usual one, like in~\cite{Sin15}, but equivalent to it; this equivalence can be deduced from, for example,~\cite[Theorem 2.18]{Br2}). We denote by~$\mathcal{Z}_q$ the image of~$\sz$ and call elements in~$\mathcal{Z}_q$ \emph{(Schlesinger--Zudilin)-\qmzv s}, abbreviated as (SZ-)\qmzv s. A remarkable property of SZ-\qmzv s is their invariance under the involution~$\tau: \QB^\circ \rightarrow \QB^\circ$, defined by~$\Q$-linearity,~$\tau(\bone) :=\bone$, and
\[\tau\left(u_{k_1} u_0^{z_1} \cdots u_{k_\dd} u_0^{z_\dd}\right) \df u_{z_\dd+1} u_0^{k_\dd-1} \cdots u_{z_1+1} u_0^{k_1-1}\]
for all~$d\in\Z_{>0}$, $k_1,\ldots,k_\dd \in\Z_{\geq 1}, \text{ and } z_1,\dots,z_\dd\in\Z_{\geq 0}$ (see \cite[Theorem 8.3]{Zh};~$\tau$ is often referred to as \emph{duality}). At this point, we highlight the following conjecture by Bachmann regarding the structure of~$\mathcal{Z}_q$ (see~\cite{BaTalk}; a published version can be found in~\cite[Conjecture 1]{Zu}).
\begin{conjecture}[Bachmann]
    \label{conj:MPallrelsqversionApp}
    All~$\Q$-linear relations among elements in~$\mathcal{Z}_q$ are obtained by the stuffle product~$\ast$ and duality~$\tau$.
\end{conjecture}

Furthermore, the space~$\mathcal{Z}_q$ contains the Fourier expansion of all quasi-modular forms via the~$q$-expansion of SZ-\qmzv s,
\begin{align*}
    \sz(\word) = \sum\limits_{N\geq 0} \psi_N(\word) q^N,
\end{align*}
where~$\psi_N(\word)$ denotes the~$N$-th Fourier coefficient of~$\sz(\word)$ for any~$\word\in\mathcal{U}^{\ast,\circ}$ (see~\cite{Ba3}). This paper provides a combinatorial framework for studying the Fourier coefficients of~$q$MZVs. Specifically, we interpret each~$\psi_N(\word)$ as a finite sum over so-called \emph{marked partitions}, as introduced in~\cite{Br2}. Moreover, we describe the stuffle product in terms of a pairing on marked partitions.

Let us now introduce marked partitions as in~\cite{Br2}. Consider a partition~$p$ of~$N$ with~$\dd$ distinct parts~$m_j$ having multiplicities~$n_j$. This means we have
\begin{align*}
    m_1>\cdots>m_\dd > 0,\quad n_1,\dots,n_\dd\in\Z_{>0},\quad N=m_1 n_1 + \cdots + m_\dd n_\dd.
\end{align*}
We use dots to mark rows and columns in the Young diagram corresponding to~$p$. If~$k_j$ rows of length~$m_j$ are marked, we call~$(k_1, \dots, k_\dd)$ the \emph{type} of the row marking. A row marking is said to be \emph{distinct} if the lowest row of each length~$m_j$ is marked. A \emph{distinct column marking} of type~$(z_1, \dots, z_\dd)$ for~$p$ is defined as a distinct row marking of type~$(z_\dd, \dots, z_1)$ for the conjugate partition of~$p$. A Young diagram, together with both a distinct row marking and a distinct column marking, is called a \emph{marked partition}. Let~$\widehat{p}$ denote a marked partition with~$(k_1, \dots, k_\dd)$ as the type of its row marking, and~$(z_1, \dots, z_\dd)$ as the type of its column marking. We then associate to the marked partition~$\widehat{p}$ the word~$u_{k_1}u_0^{z_1-1}\cdots u_{k_\dd}u_0^{z_\dd-1}\in\mathcal{U}^{\ast,\circ},$ which we call the \emph{type} of the marked partition~$\widehat{p}$. Given a marked partition~$\widehat{p}$, we will often denote the corresponding partition simply by~$p$.

\begin{example}
\label{ex:firstmp}
    The following is a marked partition of type~$\word = u_2u_0u_0u_1u_1u_0$ of the integer~$N=9\cdot 3 + 5\cdot 2 + 2\cdot 2 = 41$.
    \begin{align*}
        \crpart{9,9,9,5,5,2,2}{1,3,5,7}{1,2,5,6,7,9}
    \end{align*}
\end{example}

\begin{definition}
\label{def:markedpartitionsApp}
    \begin{enumerate}[leftmargin=*]
        \item We interpret~$\emptyset$ as the unique marked partition (of~$N=0$) of type~$\bone$.
        \item For any $\word\in\mathcal{U}^{\ast,\circ}$, we define~$\mathcal{MP}_{\word}$ as the set of all marked partitions of type~$\word$.
        \item We denote by~$\mathcal{MP}:=\bigcup\limits_{\word\in\mathcal{U}^{\ast,\circ}}\mathcal{MP}_{\word}$ the set of all marked partitions.
        \item Given a (marked) partition, we call each part of its Young diagram, obtained by slicing it horizontally below the rows containing corners, a \emph{horizontal block (of the partition/Young diagram).}
    \end{enumerate}
\end{definition}

Note that the marked partition in Example~\ref{ex:firstmp} has three horizontal blocks in the terminology of Definition~\ref{def:markedpartitionsApp}.

One has the following connection of marked partitions and the Fourier coefficient of SZ-\qmzv s.

\begin{proposition}[\cite{Br2}]
    For all~$N\in\Z_{\geq 0}$ and~$\word\in\mathcal{U}^{\ast,\circ}$, we have
    \begin{align} 
        \psi_N(\word) = \#\left\{\widehat{p}\in\mathcal{MP}_\word\, |\, p\text{ is partition of } N\right\}.
    \end{align}
\end{proposition}

To state the main theorem regarding the combinatorial interpretation of the product of SZ-$q$MZVs, we need the following pairing~$\Phi$ on the set of marked partitions.
\begin{definition}
\label{def:Phi}
    The map~$\Phi\colon \mathcal{MP} \times \mathcal{MP} \to \mathcal{MP}$ is defined as follows: Given marked partitions~$\widehat{p_1}$ of~$N_1$ and~$\widehat{p_2}$ of~$N_2$, the marked partition~$\widehat{p} = \Phi(\widehat{p_1}, \widehat{p_2})$ of~$N_1+N_2$ is obtained using the following rules:
    \begin{enumerate}[leftmargin=*]
        \item We set~$\Phi(\emptyset,\widehat{p_2}) := \widehat{p_2}$ and~$\Phi(\widehat{p_1},\emptyset) := \widehat{p_1}$.
        \item Cut the Young diagrams of~$\widehat{p_1}$ and~$\widehat{p_2}$ horizontally below the rows containing corners, separating them into their horizontal blocks. Construct the Young diagram of~$\widehat{p}$ by gluing these horizontal blocks together horizontally. If both~$\widehat{p_1}$ and~$\widehat{p_2}$ contain horizontal blocks of the same length, place the blocks from~$\widehat{p_1}$ above those from~$\widehat{p_2}$ in the new partition.
        \item In (ii), keep the markings of the rows.
        \item If there was a marking in the~$j$-th leftmost column of~$\widehat{p_1}$ or~$\widehat{p_2}$, the~$j$-th leftmost column of~$\widehat{p}$ will be marked as well.
    \end{enumerate}
\end{definition}

\begin{remark}
    Note that the map~$\Phi$ is associative but not commutative. The marked partitions~$\Phi(\widehat{p_1},\widehat{p_2})$ and~$\Phi(\widehat{p_2},\widehat{p_1})$ differ in the row markings at most.
\end{remark}

\begin{example}
Consider the following pair of marked partitions.
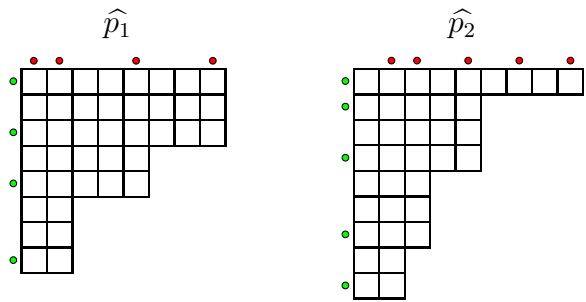
\begin{figure}[H]
\centering
\begin{tikzpicture}
    \node (b1) {\crpart{8,8,8,5,5,2,2,2}{1,3,5,8}{1,2,5,8}};
    \node (b2) at ($(b1.east) + (3,-.167)$) {\crpart{9,5,5,5,3,3,3,2,2}{1,2,4,7,9}{2,3,5,7,9}};
    \node (p1) at ($(b1.north) + (0,.3)$) {$\widehat{p_1}$};
    \node (p2) at ($(b2.north) + (0,.3)$) {$\widehat{p_2}$};
\end{tikzpicture}
\end{figure}
We slice them into their horizontal blocks.
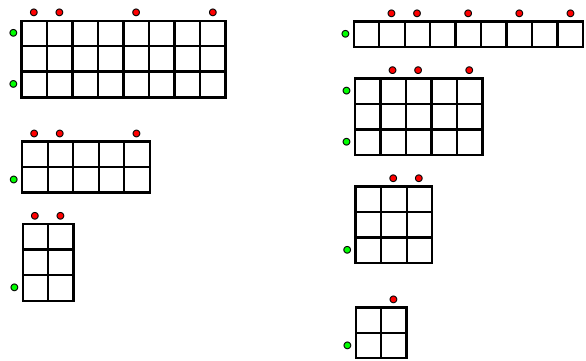
\begin{figure}[H]
\centering
\begin{tikzpicture}[every text node part/.style={align=center}]
    \node (left1) {\crpart{8,8,8}{1, 3}{1,2,5,8}};
    \node (left2) at ($(left1.south) + (-.5,-.7)$) {\crpart{5,5}{2}{1,2,5}};
    \node (left3) at ($(left2.south) + (-.5,-.7)$) {\crpart{2,2,2}{3}{1,2}};

    \node (right1) at ($(left1.east) + (3,.33)$) {\crpart{9}{1}{2,3,5,7,9}};
    \node (right2) at ($(right1.south) + (-.66,-.7)$) {\crpart{5,5,5}{1,3}{2,3,5}};
    \node (right3) at ($(right2.south) + (-.33,-.7)$) {\crpart{3,3,3}{3}{2,3}};
    \node (right4) at ($(right3.south) + (-.167,-.7)$) {\crpart{2,2}{2}{2}};
\end{tikzpicture}
\end{figure}
Following the definition of~$\Phi$, we obtain~$\Phi(\widehat{p_1}, \widehat{p_2})$ by sorting the horizontal blocks and rearranging them as described in the rules of Definition~\ref{def:Phi}.
\begin{figure}[H]
\centering
\begin{tikzpicture}
    \node (new1) {\crpart{9}{1}{2,3,5,7,9}};
    \node (new2) at ($(new1.south) + (-.167,-.7)$){\crpart{8,8,8}{1, 3}{1,2,5,8}};
    \node (new3) at ($(new2.south) + (-.5,-.7)$) {\crpart{5,5}{2}{1,2,5}};
    \node (new4) at ($(new3.south) + (0,-.7)$) {\crpart{5,5,5}{1,3}{2,3,5}};
    \node (new5) at ($(new4.south) + (-.33,-.7)$) {\crpart{3,3,3}{3}{2,3}};
    \node (new6) at ($(new5.south) + (-.167,-.7)$) {\crpart{2,2,2}{3}{1,2}};
    \node (new7) at ($(new6.south) + (0,-.7)$) {\crpart{2,2}{2}{2}};
    \node (Ctext) at ($(new1.north) + (0,.3)$) {Horizontal blocks ordered};

    \node (new) at ($0.5*(new1.east) + 0.5*(new7.east) + (4,0)$) {\crpart{9,8,8,8,5,5,5,5,5,3,3,3,2,2,2,2,2}{1,2,4,6,7,9,12,15,17}{1,2,3,5,7,8,9}};
    \node (Phi) at ($(new.north) + (0,.3)$) {$\Phi(\widehat{p_1},\widehat{p_2})$};

    \draw[-stealth] ($(new4.east)+(1,-.54)$) -- ($(new.west)+(-.3,-.3)$);
\end{tikzpicture}
\end{figure}
\end{example}

\begin{definition}
\begin{enumerate}
    \item For~$\word_1,\word_2,\word\in\mathcal{U}^{\ast,\circ}$, we set~$m_{\word_1,\word_2;\word}\in\Z_{\geq 0}$ to be the multiplicity of~$\word$ in~$\word_1\ast\word_2$, i.e., to be the unique integer satisfying
\begin{align*}
    \word_1\ast\word_2 = \sum\limits_{\word\in\mathcal{U}^{\ast,\circ}} m_{\word_1,\word_2;\word}\word.
\end{align*}
\item For~$\word_1,\word_2,\word\in\mathcal{U}^{\ast,\circ}$ and~$\widehat{p}\in\mathcal{MP}_\word$, we define
\begin{align}
    m_{\word_1,\word_2;\widehat{p}} := \#\left\{(\widehat{p_1},\widehat{p_2})\in \mathcal{MP}_{\word_1}\times \mathcal{MP}_{\word_2} \mid \Phi (\widehat{p_1},\widehat{p_2}) = \widehat{p}\right\}.
\end{align}
\end{enumerate} 
\end{definition}
Note that, for fixed~$\word_1,\word_2\in\mathcal{U}^{\ast,\circ}$, all but finitely many~$m_{\word_1,\word_2;\word}$ are zero.
\medskip 

\paragraph{\textbf{Statement of results.}}
The main result of this paper provides a combinatorial interpretation of the stuffle product in terms of marked partitions.
\begin{theorem}
\label{thm:mainMP}
    Let~$\word_1,\word_2,\word\in\mathcal{U}^{\ast,\circ}$ be words. For all~$\widehat{p}\in\mathcal{MP}_{\word}$, we have
    \begin{align}
    \label{eq:mainMP}
    m_{\word_1,\word_2;\widehat{p}} = m_{\word_1,\word_2;\word}.
    \end{align}
    In particular, given~$\word_1,\word_2$,~$m_{\word_1,\word_2;\widehat{p}}$ only depends on the word~$\word$ but not on the marked partition~$\widehat{p}\in\mathcal{MP}_{\word}$.
\end{theorem}

A remarkable aspect of Theorem~\ref{thm:mainMP} is that, conjecturally, all linear relations among Multiple~$q$-Zeta Values can now be described combinatorially using marked partitions. This follows from Conjecture~\ref{conj:MPallrelsqversionApp} and the fact that duality is already expressible in terms of marked partitions (see~\cite{Br2}). Theorem~\ref{thm:mainMP} completes this picture by providing a combinatorial interpretation of the stuffle product using marked partitions.

\begin{example}
    Consider~$\word_1 = u_1u_0u_1u_0,\, \word_2 = u_2u_0u_0$, and~$\word = u_3u_0u_0u_1u_0$. Note that~$m_{\word_1,\word_2;\word} = 4$. Furthermore, consider
    \begin{figure}[H]
\begin{tikzpicture}
    \node (b1) {\crpart{8,8,8,8,8,3,3,3}{2,3,5,8}{1,3,5,6,8}};
    \node (p1) at ($(b1.west) + (-.4,0)$) {$\widehat{p} =~$};
    \node (p2) at ($(b1.east) + (1.5,0)$) {$\in\mathcal{MP}_{\word}$.};
\end{tikzpicture}
\end{figure}
The pairs~$(\widehat{p_1},\widehat{p_2})\in \mathcal{MP}_{\word_1}\times \mathcal{MP}_{\word_2}$ satisfying~$\Phi (\widehat{p_1},\widehat{p_2}) = \widehat{p}$ are the following:
\begin{figure}[H]
\begin{tikzpicture}[every text node part/.style={align=center}]
    \node (left1) {$\crpart{8,8,3,3,3}{2, 5}{1,3,5,8},\quad \crpart{8,8,8}{1,3}{1,6,8}$};
    \draw ($(left1.west) + (0,0.9)$) arc (160:200:2.6);
    \draw ($(left1.east) + (0,-0.9)$) arc (-20:20:2.6);
    \node at ($(left1.south east) + (.2,.2)$) {,};
    \node (left2) at ($(left1.east) + (4,0)$) {$\crpart{8,8,3,3,3}{2, 5}{1,3,6,8},\quad \crpart{8,8,8}{1,3}{1,5,8}$};
    \draw ($(left2.west) + (0,0.9)$) arc (160:200:2.6);
    \draw ($(left2.east) + (0,-0.9)$) arc (-20:20:2.6);
    \node at ($(left2.south east) + (.2,.2)$) {,};
    \node (left3) at ($(left1.south) + (0,-1.3)$) {$\crpart{8,8,3,3,3}{2, 5}{1,3,5,8},\quad \crpart{8,8,8}{1,3}{5,6,8}$};
    \draw ($(left3.west) + (0,0.9)$) arc (160:200:2.6);
    \draw ($(left3.east) + (0,-0.9)$) arc (-20:20:2.6);
    \node at ($(left3.south east) + (.2,.2)$) {,};
    \node (left4) at ($(left2.south) + (0,-1.3)$) {$\crpart{8,8,3,3,3}{2, 5}{1,3,6,8},\quad \crpart{8,8,8}{1,3}{5,6,8}$};
    \draw ($(left4.west) + (0,0.9)$) arc (160:200:2.6);
    \draw ($(left4.east) + (0,-0.9)$) arc (-20:20:2.6);
    \node at ($(left4.south east) + (.2,.2)$) {.};
\end{tikzpicture}
\end{figure}
In particular, the claim of Theorem~\ref{thm:mainMP} in this case is true due to
\begin{align}
    m_{\word_1,\word_2;\widehat{p}} = 4 = m_{\word_1,\word_2;\word}.
    \end{align}
\end{example}

\medskip

\paragraph{\textbf{Organization of the paper.}} In Section~\ref{sec:stuffleprepare}, we consider a recursion of the stuffle product. This will be the key step of the main theorem in Section~\ref{sec:MPproof} where we show that the numbers~$m_{\word_1,\word_2;\widehat{p}}$ and~$m_{\word_1,\word_2;\word}$ satisfy the same recursion and, therefore, must be equal.

\medskip

\paragraph{\textbf{Acknowledgements.}}  The author thanks Henrik Bachmann, Annika Burmester, Niclas Confurius, and Ulf Kühn for valuable discussions and their comments on an earlier version of this paper. Furthermore, the author thanks the referees for their helpful comments and suggestions.

\section{About the stuffle product}
\label{sec:stuffleprepare}

The following characterization of the stuffle product is equivalent to its definition.

\begin{proposition}
\label{prop:stufflerev}
    Consider~$\word_1,\word_2\in\mathcal{U}^{\ast}$ and~$j_1,j_2\in\Z_{\geq 0}$. Then, we have
    \begin{align*}
        \word_1u_{j_1} \ast \word_2u_{j_2}  =   \left( \word_1 \ast \word_2 u_{j_2}\right) u_{j_1}+ \left(\word_1 u_{j_1} \ast \word_2\right) u_{j_2}  + \left(\word_1 \ast \word_2\right) u_{j_1+j_2}.
    \end{align*}
\end{proposition}
\begin{proof}
    The proof proceeds by induction on~$\len(\word_1) + \len(\word_2)$, where the induction step involves the recursive definition of the stuffle product.
\end{proof}

To prepare for the proof of Theorem~\ref{thm:mainMP}, the following recursion for the stuffle product is needed.

\begin{lemma}
    \label{lem:strec0}
     Let~$\word_1',\, \word_2'\in\mathcal{U}^{\ast}$,~$j_1,j_2\in\Z_{>0}$ be words, and~$n_1,n_2\in\Z_{\geq 0}$. Consider
    \begin{align*}
        \word_1 =  \word_1'u_{j_1}u_0^{n_1},\quad \word_2 =  \word_2'u_{j_2}u_0^{n_2}.
    \end{align*}
    We have
    \begin{align*}
        \word_1\ast\word_2 = & \sum\limits_{\substack{0\leq k\leq j\leq n_2\\ 0\leq\varepsilon\leq \min\{1,n_2-j\}}} \binom{n_1+k}{n_1}\binom{n_1}{j-k}  \left(\word_1'\ast \word_2'u_{j_2}u_0^{n_2-j-\varepsilon}\right)u_{j_1}u_0^{n_1+k}
        \\
        &+\sum\limits_{\substack{0\leq k\leq j\leq n_1\\ 0\leq\varepsilon\leq \min\{1,n_1-j\}}} \binom{n_2+k}{n_2}\binom{n_2}{j-k} \left(\word_1'u_{j_1}u_0^{n_1-j-\varepsilon} \ast\word_2'\right)u_{j_2}u_0^{n_2+k}
        \\
        &+\sum\limits_{k=0}^{n_2} \binom{n_1+k}{n_1} \binom{n_1}{n_2-k} (\word_1'\ast\word_2')u_{j_1+j_2}u_0^{n_1+k}.
    \end{align*}
\end{lemma}

\begin{proof}
    For $n_1 = 0$, the proof is by induction on~$n_2$. The base case~$n_1 = n_2 = 0$ corresponds to Proposition~\ref{prop:stufflerev}. In the induction step, we first use Proposition~\ref{prop:stufflerev} and then the induction hypothesis to establish the claim. The case~$n_2 = 0$ is handled analogously. For~$n_1, n_2 > 0$, the proof follows by induction on~$n_1 + n_2$, using Proposition~\ref{prop:stufflerev} and appropriately rearranging sums and applying identities for binomial coefficients.
\end{proof}

We write
\begin{align*}
    \delta_{\bullet} := \begin{cases}
        1,\quad \text{if }\bullet\text{ is true,}\\
        0,\quad \text{if }\bullet\text{ is false}
    \end{cases}
\end{align*}
for the Kronecker Delta, as usual. From Lemma~\ref{lem:strec0}, we derive the following recursion for the numbers~$m_{\word_1, \word_2; \word}$.

\begin{proposition}
\label{prop:mwwwrecursion}
    Let~$\word_1,\word_2,\word\in\mathcal{U}^{\ast,\circ}$ be words. \begin{enumerate}
        \item If~$\word_1 = \bone$, we have~$m_{\word_1,\word_2;\word} = \delta_{\word = \word_2}$.
        \item If~$\word_2 = \bone$, we have~$m_{\word_1,\word_2;\word} = \delta_{\word = \word_1}$.
        \item If~$\word = \bone$, we have~$m_{\word_1,\word_2;\word} = \delta_{\word_1=\word_2=\bone}$.
        \item If~$\word_1,\word_2,\word\neq\bone$, write 
        \begin{align*}
        \word_1 =  \word_1'u_{j_1}u_0^{n_1},\quad \word_2 =  \word_2'u_{j_2}u_0^{n_2},\quad \word = \word'u_{j_3} u_0^{n_3}
    \end{align*}
    with unique~$\word_1',\word_2',\word'\in\mathcal{U}^{\ast,\circ}$,~$j_1,j_2,j_3\in\Z_{>0}$, and~$n_1,n_2,n_3\in\Z_{\geq 0}$. Then,
    \begin{align*}
        m_{\word_1,\word_2;\word} 
        =\, & \sum\limits_{\substack{0\leq k\leq j\leq n_2\\ 0\leq\varepsilon\leq \min\{1,n_2-j\}}} \binom{n_1+k}{n_1}\binom{n_1}{j-k}   m_{\word_1',\word_2'u_{j_2}u_0^{n_2-j-\varepsilon};\word'}\delta_{\substack{j_1 = j_3}}\delta_{n_1+k = n_3}
        \\
        &\, +\sum\limits_{\substack{0\leq k\leq j\leq n_1\\ 0\leq\varepsilon\leq \min\{1,n_1-j\}}} \binom{n_2+k}{n_2}\binom{n_2}{j-k}  m_{\word_1'u_{j_1}u_0^{n_1-j-\varepsilon},\word_2';\word'}\delta_{\substack{j_2=j_3}}\delta_{n_2+k = n_3} 
        \\
        &\, +\sum\limits_{k=0}^{n_2} \binom{n_1+k}{n_1} \binom{n_1}{n_2-k} m_{\word_1',\word_2';\word'}\delta_{\substack{j_1+j_2=j_3}}\delta_{n_1+k=n_3}.
    \end{align*}
    \end{enumerate}
\end{proposition}
\begin{proof}
    While (i), (ii), and (iii) are evident following the definition of the stuffle product, (iv) is an immediate consequence of Lemma~\ref{lem:strec0}.
\end{proof}

\section{Proof of our main theorem}
\label{sec:MPproof}
Having established the recursion for the stuffle product in Lemma~\ref{lem:strec0}, we can now proceed to prove our main theorem. The underlying idea is straightforward: We demon- strate, in a combinatorial manner, that the numbers~$m_{\word_1,\word_2;\widehat{p}}$ satisfy the same recursion as the numbers~$m_{\word_1,\word_2;\word}$ (where~$\word$ represents the type of~$\widehat{p}$). This allows the claim to follow by induction on~$\dep(\word)$. To facilitate our combinatorial arguments in the proof of the main theorem, we first introduce some necessary notions.

\begin{definition}
\label{def:mpnotation}
    Let~$\word \in \mathcal{U}^{\ast,\circ}$ and~$\widehat{p} \in \mathcal{MP}_\word$. \begin{enumerate} \item We denote by~$\smpart{\widehat{p}}$ the minimal length of the parts in~$\widehat{p}$.

\item We denote by~$\cmark{\widehat{p}} \subset \{1, \dots, \smpart{\widehat{p}}\}$ the column markings of~$\widehat{p}$ that occur in the horizontal block of minimal length in~$\widehat{p}$. More precisely,~$j \in \cmark{\widehat{p}}$ if and only if the~$j$-th leftmost column in~$\widehat{p}$ has a marking and~$j \leq \smpart{\widehat{p}}$.

\item We denote by~$(\widehat{p})_{-1}$ the marked partition obtained from~$\widehat{p}$ by removing its horizontal block of minimal length~$\smpart{\widehat{p}}$, along with the corresponding row markings. Additionally, any column markings in the~$j$-th leftmost column are removed if~$j \in \cmark{\widehat{p}}$.

\item We denote by~$(\widehat{p})_{1}$ the marked partition consisting of the horizontal block of minimal length~$\smpart{\widehat{p}}$ from~$\widehat{p}$. This includes preserving the row markings, and column markings are retained in the~$j$-th leftmost column if and only if~$j \in \cmark{\widehat{p}}$. 
\end{enumerate}
\end{definition}
Note that we will sometimes use the phrase \lq\lq $m$ is a column marking of~$\widehat{p}$\rq\rq to mean that the~$m$-th leftmost column of~$\widehat{p}$ is marked.

Let us now consider an example to illustrate the notation introduced in Definition~\ref{def:mpnotation}.

\begin{example}
Consider
    \begin{figure}[H]
    \centering
\begin{tikzpicture}
    \node (b1) {\crpart{9,8,8,5,5,2,2,2}{1,3,5,8}{1,2,5,8,9}.};
    \node () at ($(b1.west) + (-0.4,0)$){$\widehat{p} = $};
\end{tikzpicture}
\end{figure}
Using the notation from Definition~\ref{def:mpnotation}, we have~$\smpart{\widehat{p}} = 2$ and~$\cmark{\widehat{p}} = \{1,2\}$. Furthermore, we have

\begin{figure}[H]
\centering
\begin{tikzpicture}
    \node (b1) {\crpart{9,8,8,5,5}{1,3,5}{5,8,9}};
    \node () at ($(b1.west)+(-0.7,0)$) {$(\widehat{p})_{-1} = $};
    \node (b2) at ($(b1.east) + (4,-.167)$) {\crpart{2,2,2}{3}{1,2}.};
    \node () at ($(b2.west)+(-0.5,0)$) {$(\widehat{p})_{1} = $};
    \node () at ($0.5*(b1) + 0.5*(b2)$) {and};
\end{tikzpicture}
\end{figure}
\end{example}

\begin{remark}
\begin{enumerate}
    \item For all~$\word\in\mathcal{U}^{\ast,\circ}$ and~$\widehat{p}\in\mathcal{MP}_\word$, we have
    \begin{align}
        \widehat{p} = \Phi((\widehat{p})_{-1},(\widehat{p})_1).
    \end{align}
    \item Let~$\word\in\mathcal{U}^{\ast,\circ}$ with~$\dep(\word)\geq 1$. Write~$\word = \word'u_j u_0^n$ with~$\word'\in\mathcal{U}^{\ast,\circ}$,~$j\in\Z_{>0}$, and~$n\in\Z_{\geq 0}$ uniquely determined. For all~$\widehat{p}\in\mathcal{MP}_{\word}$, we have~$(\widehat{p})_{-1}\in\mathcal{MP}_{\word'}$.
    \end{enumerate}
\end{remark}

With the additional notation introduced in Definition~\ref{def:mpnotation}, we are now prepared to prove our main theorem, which asserts that~$\Phi$ describes the stuffle product at the level of marked partitions.


\begin{proof}[Proof of Theorem~\ref{thm:mainMP}]
    We begin with the three special cases~$\word_1 =\bone$,~$\word_2 = \bone$,~$\word = \bone$. If~$\word_1 = \bone$, we note that
    \begin{align}
        \word_1\ast\word_2 = \bone\ast\word_2 = \word_2.
    \end{align}
    Therefore, for all~$\word\in\mathcal{U}^{\ast,\circ}$, by Proposition~\ref{prop:mwwwrecursion}(i), we have
    \begin{align*}
        m_{\bone,\word_2;\word} = \delta_{\word = \word_2}.
    \end{align*}
    Furthermore, we have~$\mathcal{MP}_{\word_1} = \mathcal{MP}_\bone = \{\emptyset\}$. This implies for all~$\word\in\mathcal{U}^{\ast,\circ}$ and~$\widehat{p}\in\mathcal{MP}_\word$ that
    \begin{align}
        m_{\word_1,\word_2;\widehat{p}} =&\, \#\left\{(\widehat{p_1},\widehat{p_2})\in \mathcal{MP}_{\word_1}\times \mathcal{MP}_{\word_2} \mid \Phi (\widehat{p_1},\widehat{p_2}) = \widehat{p}\right\}
        \\
        =&\, \#\left\{(\emptyset,\widehat{p_2})\in \mathcal{MP}_{\bone}\times \mathcal{MP}_{\word_2} \mid \Phi (\emptyset,\widehat{p_2}) = \widehat{p}\right\}
        \\
        =&\, \#\left\{(\emptyset,\widehat{p_2})\in \mathcal{MP}_{\bone}\times \mathcal{MP}_{\word_2} \mid \widehat{p_2} = \widehat{p}\right\}
        \\
        =&\, \delta_{\word = \word_2}
        \\
        =&\, m_{\word_1,\word_2;\word}.
    \end{align}
    So, if~$\word_1 = \bone$, the claim follows. Similarly (using Proposition~\ref{prop:mwwwrecursion}(ii)), we obtain the claim for~$\word_2 = \bone$. Next, consider the case of~$\word_1,\word_2\in\mathcal{U}^{\ast,\circ}$ arbitrary and~$\word = \bone$. Then,
    \begin{align*}
        m_{\word_1,\word_2;\bone} = \delta_{\word_1 = \word_2 = \bone}.
    \end{align*}
    Furthermore, we have~$\mathcal{MP}_\word = \mathcal{MP}_\bone = \{\emptyset\}$. This implies that
    \begin{align}
        m_{\word_1,\word_2;\emptyset} =&\, \#\left\{(\widehat{p_1},\widehat{p_2})\in \mathcal{MP}_{\word_1}\times \mathcal{MP}_{\word_2} \mid \Phi (\widehat{p_1},\widehat{p_2}) = \emptyset\right\}
        \\
        =&\, \#\left\{(\widehat{p_1},\widehat{p_2})\in \mathcal{MP}_{\word_1}\times \mathcal{MP}_{\word_2} \mid \widehat{p_1} = \widehat{p_2} = \emptyset\right\}
        \\
        =&\, \delta_{\word_1 = \word_2 = \bone}
        \\
        =&\, m_{\word_1,\word_2;\bone},
    \end{align}
    where the last step follows from Proposition~\ref{prop:mwwwrecursion}(iii). Therefore, the claim follows also for all~$\word_1,\word_2\in\mathcal{U}^{\ast,\circ}$ when~$\word=\bone$.

    Therefore, we may assume that~$\word_1, \word_2, \word \neq \bone$ in the following. We write: \begin{align*} \word_1 = \word_1' u_{j_1} u_0^{n_1}, \quad \word_2 = \word_2' u_{j_2} u_0^{n_2}, \quad \word = \word' u_{j_3} u_0^{n_3}, \end{align*} where~$\word_1', \word_2', \word' \in \mathcal{U}^{\ast,\circ}$,~$j_1, j_2, j_3 \in \Z_{>0}$, and~$n_1, n_2, n_3 \in \Z_{\geq 0}$ are uniquely determined.

We prove the claim of the theorem by induction on~$\dep(\word)$. Note that the base case~$\dep(\word) = 0$ has already been proven, as in that case~$\word = \bone$. Therefore, we assume~$\dep(\word) > 0$ and that the claim holds for all smaller values of~$\dep(\word)$.

Let~$\widehat{p} \in \mathcal{MP}_\word$ be arbitrary. In particular, in~$(\widehat{p})_1$, there are exactly~$j_3$ rows marked and~$n_3 + 1$ columns, including the row and column that contain the corner. To determine the number~$m_{\word_1, \word_2; \widehat{p}}$, we need to count the pairs~$(\widehat{p_1}, \widehat{p_2}) \in \mathcal{MP}_{\word_1} \times \mathcal{MP}_{\word_2}$ of marked partitions such that~$\Phi(\widehat{p_1}, \widehat{p_2}) = \widehat{p}$. Specifically, we have
    \begin{align*}
        \left(\Phi(\widehat{p_1},\widehat{p_2})\right)_1 = (\widehat{p})_1.
    \end{align*}
    We consider three distinct cases:
    \begin{enumerate}
        \item $\smpart{\widehat{p_1}} < \smpart{\widehat{p_2}}$: The horizontal block of minimal length in~$\widehat{p_1}$ (ignoring column markings) constitutes the entire horizontal block of minimal length in~$\widehat{p}$ (also ignoring column markings).
        \item $\smpart{\widehat{p_1}} > \smpart{\widehat{p_2}}$: The horizontal block of minimal length in~$\widehat{p_2}$ (ignoring column markings) constitutes the entire horizontal block of minimal length in~$\widehat{p}$ (also ignoring column markings).
        \item $\smpart{\widehat{p_1}} = \smpart{\widehat{p_2}}$: The horizontal block of minimal length in~$\widehat{p}$ is formed by placing the horizontal block of minimal length in~$\widehat{p_1}$ above the horizontal block of minimal length in~$\widehat{p_2}$. In particular, we have~$(\widehat{p})_1 = \Phi((\widehat{p_1})_1, (\widehat{p_2})_1)$.
    \end{enumerate}
        \paragraph{\textbf{Case $\smpart{\widehat{p_1}} < \smpart{\widehat{p_2}}$.}} We aim to determine the number
    \begin{align*}
        \#\left\{(\widehat{p_1},\widehat{p_2})\in \mathcal{MP}_{\word_1}\times \mathcal{MP}_{\word_2} \mid \smpart{\widehat{p_1}} < \smpart{\widehat{p_2}},\ \Phi (\widehat{p_1},\widehat{p_2}) = \widehat{p}\right\}.
    \end{align*}
    Note that for~$(\widehat{p_1},\widehat{p_2})\in \mathcal{MP}_{\word_1}\times \mathcal{MP}_{\word_2}$ with~$\smpart{\widehat{p_1}} < \smpart{\widehat{p_2}}$ and~$\Phi(\widehat{p_1},\widehat{p_2}) = \widehat{p}$, the following holds:
    \begin{align*}
        \widehat{p} = \Phi(\widehat{p_1},\widehat{p_2}) = \Phi((\Phi(\widehat{p_1},\widehat{p_2}))_{-1},(\widehat{p})_1).
    \end{align*}
    Furthermore, in this case, we have
    \begin{align*}
        (\Phi(\widehat{p_1},\widehat{p_2}))_{-1} = \Phi\left((\widehat{p_1})_{-1},\widetilde{\widehat{p_2}}\right)\in\mathcal{MP}_{\word'},
    \end{align*}
    where~$\widetilde{\widehat{p_2}}$ is the marked partition~$\widehat{p_2}$ with the column markings in~$\cmark{\widehat{p}} \backslash \{\smpart{\widehat{p}}\}$ removed. Note that we have
    \begin{align*}
        (\widehat{p_1})_{-1}\in\mathcal{MP}_{\word_1'}\quad\text{and}\quad \widetilde{\widehat{p_2}}\in\mathcal{MP}_{\word_2'u_{j_2}u_0^{n_2-j-\varepsilon}},
    \end{align*}
    where~$j= \#\left((\cmark{\widehat{p}}\backslash\{\smpart{\widehat{p}}\})\cap\cmark{\widehat{p_2}}\right)$ and
    \begin{align*}
        \varepsilon = \begin{cases}
            0,\text{ if } \smpart{\widehat{p}}\not\in \cmark{\widehat{p_2}},
            \\
            1,\text{ if } \smpart{\widehat{p}}\in \cmark{\widehat{p_2}}.
        \end{cases}
    \end{align*}
    Now, fix~$0\leq j\leq n_2$,~$0\leq\varepsilon\leq\min\{1,n_2-j\}$ and marked partitions
    \begin{align*}
        \widehat{q_1}\in\mathcal{MP}_{\word_1'}\quad\text{and}\quad \widetilde{\widehat{q_2}}\in\mathcal{MP}_{\word_2'u_{j_2}u_0^{n_2-j-\varepsilon}}
    \end{align*}
    such that~$\smpart{\widehat{p}}<\smpart{\widetilde{\widehat{q_2}}}$ and
    \begin{align*}
        \widehat{q} := \Phi\left(\widehat{q_1},\widetilde{\widehat{q_2}}\right) = (\widehat{p})_{-1}\in\mathcal{MP}_{\word'}.
    \end{align*}
    Hence, we need to determine the number of pairs~$(\widehat{q}',\widehat{q_2})\in\mathcal{MP}_{u_{j_1}u_0^{n_1}}\times\mathcal{MP}_{\word_2}$ such that~$\smpart{\widehat{q}'} < \smpart{\widehat{q_2}}$,
    \begin{align*}
        \widehat{p} = \Phi(\Phi(\widehat{q_1},\widehat{q}'),\widehat{q_2}),
    \end{align*}
    and such that~$\widehat{q_2}$ without column markings in~$\cmark{\widehat{p}}\backslash\{\smpart{\widehat{p}}\}$ equals~$\widetilde{\widehat{q_2}}$ (in accordance with the notation above).
    Note that the underlying Young diagram of~$\widehat{q}'$ consists of exactly one horizontal block and is uniquely determined by~$(\widehat{p})_1$ together with the row markings of~$\widehat{q}'$. This condition is feasible only when~$j_1 = j_3$, which additionally implies~$\smpart{\widehat{q}'} = \smpart{\widehat{p}}$. Furthermore, since
    \begin{align*}
        \#(\cmark{\widehat{q}'}\backslash\{\smpart{\widehat{q}'}\}) = n_1,\quad \#(\cmark{\widehat{p}}\backslash\{\smpart{p}\}) = n_3,\quad\text{and}\quad\cmark{\widehat{q}'}\subset\cmark{\widehat{p}},
    \end{align*}
    we have~$\binom{n_3}{n_1}$ choices to determine~$\cmark{\widehat{q}'}$ and, consequently,~$\widehat{q}'$. Now, for fixed~$\widehat{q}'$, the marked partition~$\widehat{q_2}$ is determined up to the column markings in~$\cmark{\widehat{p}}$ by definition. The remaining~$n_3 - n_1$ column markings of~$\widehat{p}$ in~$\cmark{\widehat{p}} \backslash \{\smpart{\widehat{p}}\}$ must correspond to column markings of~$\widehat{q_2}$. Moreover, by the definition of~$\widetilde{\widehat{q_2}}$, the~$j - (n_3 - n_1)$ column markings of~$\widehat{q_2}$ that do not belong to~$\widetilde{\widehat{q_2}}$ must also be column markings of~$\widehat{q}'$, distinct from~$\smpart{\widehat{p}}$. This can be achieved in~$\binom{n_1}{j - (n_3 - n_1)}$ ways, which finally determines~$\widehat{q_2}$. Thus, we have proven that
    \begin{align}
        &\,\#\left\{(\widehat{p_1},\widehat{p_2})\in \mathcal{MP}_{\word_1}\times \mathcal{MP}_{\word_2} \mid \smpart{\widehat{p_1}} < \smpart{\widehat{p_2}},\ \Phi (\widehat{p_1},\widehat{p_2}) = \widehat{p}\right\}
        \\
        =&\, \sum\limits_{\substack{0\leq j\leq n_2\\ 0\leq\varepsilon\leq\min\{1,n_2-j\}}}\binom{n_3}{n_1}\binom{n_1}{j-(n_3-n_1)} m_{\word_1',\word_2'u_{j_2}u_0^{n_2-j-\varepsilon};(\widehat{p})_{-1}}\delta_{\substack{j_1 = j_3}}
        \\
        \label{eq:MPproof1}
        =&\, \sum\limits_{\substack{0\leq j\leq n_2\\ 0\leq\varepsilon\leq\min\{1,n_2-j\}}}\binom{n_3}{n_1}\binom{n_1}{j-(n_3-n_1)} m_{\word_1',\word_2'u_{j_2}u_0^{n_2-j-\varepsilon};\word'}\delta_{\substack{j_1 = j_3}},
    \end{align}
    where the last step follows from the induction step since~$\dep(\word') = \dep(\word) - 1$.

    \paragraph{\textbf{Case $\smpart{\widehat{p_1}} > \smpart{\widehat{p_2}}$.}} Analogously, we obtain
    \begin{align}
        &\,\#\left\{(\widehat{p_1},\widehat{p_2})\in \mathcal{MP}_{\word_1}\times \mathcal{MP}_{\word_2} \mid \smpart{\widehat{p_1}} > \smpart{\widehat{p_2}},\ \Phi (\widehat{p_1},\widehat{p_2}) = \widehat{p}\right\}
        \\
        \label{eq:MPproof2}
        =&\, \sum\limits_{\substack{0\leq j\leq n_1\\ 0\leq\varepsilon\leq\min\{1,n_1-j\}}}\binom{n_3}{n_2}\binom{n_2}{j-(n_3-n_2)} m_{\word_1'u_{j_1}u_0^{n_1-j-\varepsilon},\word_2';\word'}\delta_{\substack{j_2 = j_3}}.
    \end{align}
    \paragraph{\textbf{Case $\smpart{\widehat{p_1}} = \smpart{\widehat{p_2}}$.}} We aim to determine the number
    \begin{align*}
        \#\left\{(\widehat{p_1},\widehat{p_2})\in \mathcal{MP}_{\word_1}\times \mathcal{MP}_{\word_2} \mid \smpart{\widehat{p_1}} = \smpart{\widehat{p_2}},\ \Phi (\widehat{p_1},\widehat{p_2}) = \widehat{p}\right\}.
    \end{align*}
    Note that for~$(\widehat{p_1},\widehat{p_2})\in \mathcal{MP}_{\word_1}\times \mathcal{MP}_{\word_2}$ with~$\smpart{\widehat{p_1}} = \smpart{\widehat{p_2}}$ and~$\Phi(\widehat{p_1},\widehat{p_2}) = \widehat{p}$, the following holds:
    \begin{align*}
        (\widehat{p})_{-1} = \Phi((\widehat{p_1})_{-1},(\widehat{p_2})_{-1})\quad\text{and}\quad (\widehat{p})_1 = \Phi((\widehat{p_1})_1,(\widehat{p_2})_1).
    \end{align*}
    Furthermore, for~$h\in\{1,2\}$, we have
    \begin{align*}
        (\widehat{p_h})_{-1}\in\mathcal{MP}_{\word_h'},\quad (\widehat{p_h})_1\in\mathcal{MP}_{u_{j_h}u_0^{n_h}},\quad (\widehat{p})_{-1}\in\mathcal{MP}_{\word'},\quad\text{and}\quad (\widehat{p})_1\in\mathcal{MP}_{u_{j_3}u_0^{n_3}}.
    \end{align*}
    Now, for~$h\in\{1,2\}$, fix marked partitions~$\widehat{q_h}\in\mathcal{MP}_{\word_h'}$ such that
    \begin{align*}
        \Phi(\widehat{q_1},\widehat{q_2}) =\, (\widehat{p})_{-1}.
    \end{align*}
    We are interested in the number of pairs~$(\widehat{q_1}',\widehat{q_2}')\in\mathcal{MP}_{u_{j_1}u_0^{n_1}}\times\mathcal{MP}_{u_{j_2}u_0^{n_2}}$ such that
    \begin{align*}
        \Phi(\widehat{q_1}',\widehat{q_2}') = (\widehat{p})_1.
    \end{align*}
    Note that the underlying Young diagram of both~$\widehat{q_1}'$ and~$\widehat{q_2}'$ are uniquely determined by~$(\widehat{p})_1$. Also the row markings of~$\widehat{q_1}'$ and~$\widehat{q_2}'$ are determined by~$(\widehat{p})_1$. In particular, this is possible only if~$j_3 = j_1+j_2$, which implies~$\smpart{\widehat{q_1}'} =\smpart{\widehat{q_2}'} = \smpart{\widehat{p}}$. 
    
    Now, since
    \begin{align*}
        \#(\cmark{\widehat{q_1}'}\backslash\{\smpart{\widehat{q_1}'}\}) = n_1,\quad \#(\cmark{\widehat{p}}\backslash\{\smpart{p}\}) = n_3,\quad\text{and}\quad \cmark{\widehat{q_1}'}\subset\cmark{\widehat{p}},
    \end{align*}
    we have~$\binom{n_3}{n_1}$ ways to determine the column markings of~$\widehat{q_1}'$ which determines~$\widehat{q_1}'$. Furthermore, since~$\cmark{\widehat{q_1}'} \cup \cmark{\widehat{q_2}'} = \cmark{(\widehat{p})_1}$, it follows that~$n_2 - (n_3 - n_1)$ column markings of~$\widehat{q_2}'$, distinct from~$\smpart{\widehat{q_2}'}$, also belong to~$\widehat{q_1}'$. This is possible in~$\binom{n_1}{n_2 - (n_3 - n_1)}$ ways when~$\widehat{q_1}'$ is already determined, thereby determining~$\widehat{q_2}'$.
    
    Hence, by using~$\widehat{p} = \Phi((\widehat{p})_{-1},(\widehat{p})_1)$, we have proven
    
    \begin{align}
        &\,\#\left\{(\widehat{p_1},\widehat{p_2})\in \mathcal{MP}_{\word_1}\times \mathcal{MP}_{\word_2} \mid \smpart{\widehat{p_1}} = \smpart{\widehat{p_2}},\ \Phi (\widehat{p_1},\widehat{p_2}) = \widehat{p}\right\}
        \\
        =&\, \binom{n_3}{n_1}\binom{n_1}{n_2-(n_3-n_1)}m_{\word_1',\word_2';(\widehat{p})_{-1}}\delta_{\substack{j_3 = j_1+j_2}}
        \\
        \label{eq:MPproof3}
        =&\, \binom{n_3}{n_1}\binom{n_1}{n_2-(n_3-n_1)}m_{\word_1',\word_2';\word'}\delta_{\substack{j_3 = j_1+j_2}},
    \end{align}
    where the last step follows from the induction hypothesis since~$\dep(\word') = \dep(\word) - 1$.

    \paragraph{\textbf{Conclusion.}} Now,~\eqref{eq:MPproof1},~\eqref{eq:MPproof2}, and~\eqref{eq:MPproof3} imply by definition of~$m_{\word_1,\word_2;\widehat{p}}$ that
    \begin{align*}
        &\,m_{\word_1,\word_2;\widehat{p}}
        \\
        =&\, \#\left\{(\widehat{p_1},\widehat{p_2})\in \mathcal{MP}_{\word_1}\times \mathcal{MP}_{\word_2} \mid \smpart{\widehat{p_1}} < \smpart{\widehat{p_2}},\ \Phi (\widehat{p_1},\widehat{p_2}) = \widehat{p}\right\}
        \\
        &\, + \#\left\{(\widehat{p_1},\widehat{p_2})\in \mathcal{MP}_{\word_1}\times \mathcal{MP}_{\word_2} \mid \smpart{\widehat{p_1}} > \smpart{\widehat{p_2}},\ \Phi (\widehat{p_1},\widehat{p_2}) = \widehat{p}\right\}
        \\
        &\, +\#\left\{(\widehat{p_1},\widehat{p_2})\in \mathcal{MP}_{\word_1}\times \mathcal{MP}_{\word_2} \mid \smpart{\widehat{p_1}} = \smpart{\widehat{p_2}},\ \Phi (\widehat{p_1},\widehat{p_2}) = \widehat{p}\right\}
        \\
        =&\, \sum\limits_{\substack{0\leq k\leq j\leq n_2\\ 0\leq\varepsilon\leq \min\{1,n_2-j\}}} \binom{n_1+k}{n_1}\binom{n_1}{j-k}   m_{\word_1',\word_2'u_{j_2}u_0^{n_2-j-\varepsilon};\word'} \delta_{\substack{j_1 = j_3}}\delta_{n_1+k = n_3}
        \\
        &\,+\sum\limits_{\substack{0\leq k\leq j\leq n_1\\ 0\leq\varepsilon\leq \min\{1,n_1-j\}}} \binom{n_2+k}{n_2}\binom{n_2}{j-k}  m_{\word_1'u_{j_1}u_0^{n_1-j},\word_2';\word'} \delta_{\substack{j_2=j_3}}\delta_{n_2+k = n_3} 
        \\
        &\,+\sum\limits_{k=0}^{n_2} \binom{n_1+k}{n_1} \binom{n_1}{n_2-k} m_{\word_1',\word_2';\word'}\delta_{\substack{j_1+j_2=j_3}}\delta_{n_1+k=n_3}
        \\
        =&\, m_{\word_1,\word_2;\word}, 
    \end{align*}
    where the last step immediately follows from Proposition~\ref{prop:mwwwrecursion}(iv). This completes the induction step, so the theorem is proven.
\end{proof}

\begin{remark}
    Marked partitions appear to be a powerful tool for studying the coefficients in the~$q$-expansion of~$q$MZVs. They provide a combinatorial interpretation of the (algebraic) behavior of~\qmzv s. In this paper, we used them to describe the stuffle product. In~\cite{Br2}, we previously applied them to study duality in the Schlesinger–Zudilin model of~\qmzv s. Consequently, by Conjecture~\ref{conj:MPallrelsqversionApp}, we can now conjecturally describe all~$\Q$-linear relations among~\qmzv s at the level of marked partitions.

    For future work, it would be interesting to make further progress on proving problems related to the algebraic structure of~\qmzv s using marked partitions. One notable open problem is a conjecture by Bachmann~(\cite[Conjecture 4.3]{Ba1}), which asserts that the space~$\mathcal{Z}_q$ is spanned by~$\left\{u_{k_1}\cdots u_{k_\dd}\, |\, k_1,\dots,k_\dd \in\Z_{>0}\right\}.$
\end{remark}

\bibliographystyle{alpha}
\bibliography{example.bib}
\end{document}